\documentclass{amsart}
\usepackage{amsfonts}
\usepackage{amssymb}

\newtheorem{thm}{Theorem}
 
 \newtheorem{lem}[thm]{Lemma}
 \newtheorem{prop}[thm]{Proposition}
\theoremstyle{definition}
\newtheorem{remark}[thm]{Remark}

 \providecommand{\seq}[1]{\left<#1\right>}
\providecommand{\norm}[1]{\left\Vert#1\right\Vert}

\begin{document}

   \title[ On determining the domain of the adjoint operator]
   {On determining the domain of the adjoint operator.}
   \author{ Micha\l{} Wojtylak}
   \address{
   Institute of Mathematics\\
   Jagiellonian University\\
   \L ojasiewicza 6\\
   30-348 Krak\'ow\\
   Poland}
   \email{michal.wojtylak@gmail.com
}

\begin{abstract}

A theorem that is of aid in computing the domain of the adjoint operator is provided. It may
serve e.g. as a criterion for selfadjointness of a symmetric operator, for normality of a
formally normal operator or for $H$--selfadjointness of an  $H$--symmetric operator.
Differential operators and operators given by an infinite matrix  are considered as examples.
\end{abstract}
\thanks{The author gratefully acknowledges the assistance of the Polish Ministry of Higher Education and Science grant NN201 546438. }

\maketitle

\section*{Introduction}
 In the literature there exist many criteria for selfadjointness of symmetric operators. As a root of the present research one should mention the paper by Driessler and Summers \cite{driesum}, which presents a criterion for selfadjointness  connected with the notion of domination (relative boundedness) and the first commutator. Later on that result has been
 extended by Cicho\'n, Stochel and Szafraniec (\cite{dominate2}) and by the author of the present paper  (\cite{Wojtylak07,Wojtylak08}). The aim of this note is to
generalize this result in such way that it serves simultaneously as a criterion for normality
of a  formally normal operator as well as a criterion for selfadjointness of symmetric
operator in a Krein space. Furthermore, an important issue will be  illustrating this
generalization with various examples.

Let us describe now the framework of the present research. Given a pair $(A,A_0)$ of operators
in a Hilbert space, with $A$ closable and densely defined and $A_0\subseteq A^*$, we want to
provide a necessary condition for the equality $\bar{ A_0}=A^*$ . This condition should not
involve the operator $A^*$ itself but the operators $A$ and $A_0$ only. The main interest will
lie in the following instances.
\begin{itemize}
\item[(a0)] $A$ is a symmetric operator, $A_0=A$;
\item[(a1)] $A$ is a formally normal operator, $A_0=A^*\!\!\mid_{\mathcal{D}( A)}$;
\item[(a2)] $A$ is a $q$--formally normal operator  with $q\in(0,\infty)$, $A_0=A^*\!\!\mid_{\mathcal{D}(
A)}$;
\item[(a3)]  $\mathcal{D}(A)\subseteq\mathcal{D}(A^*)$ and the graph norms of $A$ and $A^*$ are equivalent
on $\mathcal{D}(A)$, $A_0=A^*\!\!\mid_{\mathcal{D}(A)}$;
\item[(a4)] $A$ is a $H$--symmetric operator,  where $H\in\mathbf{B}(\mathcal{K})$ is
selfadjoint and boundedly invertible,  $A_0=HAH^{-1}$.
\end{itemize}
Note that the equality $\bar A_0=A^*$  means in the above cases that, respectively, $\bar A$
is selfadjoint, $\bar A$ is  normal, $\bar A$ is $q$--normal, $\mathcal{D}(\bar
A)=\mathcal{D}(A^*)$, and $\bar A$ is $H$--selfadjoint (see the Preliminaries for this and for
definitions of the classes appearing above). After these explanations we can present the main
result of the paper. The theorem is proved later on in a slightly stronger form as Theorem
\ref{main4}, cf. Remark \ref{r1}. If  $S$ is an operator in a Hilbert space $\mathcal{K}$ then
$\mathcal{D}( S)$ and $\mathcal{R}(S)$ denote, respectively,  the domain and the range of $S$
and ${\mathop{\rm WOT\ lim}}$ stands for the limit in the weak operator topology.
\begin{thm}\label{intro}
Let  $A$ be a closable, densely defined operator in a Hilbert space $\mathcal{K}$ and let
$A_0\subseteq A^*$. If there exists a sequence
$(T_n)_{n=0}^\infty\subseteq\mathbf{B}(\mathcal{K})$ such that
$$
{\mathop{\rm WOT\ lim}_{n\to\infty}}\  T_n=I_\mathcal{K},
$$
\begin{equation}\label{inclran}
\mathcal{R}(T_n)\subseteq\mathcal{D}(\bar A),\qquad
\mathcal{R}(T_n^*)\subseteq\mathcal{D}({\bar A_0}),\qquad n\in\mathbb{N}
\end{equation}
and
\begin{equation}\label{komintro}
\sup_{n\in\mathbb{N}}\norm{\bar AT_n-T_n\bar A}<+\infty,
\end{equation}
 then $\bar A_0=A^*$.
\end{thm}

In the classical literature like  \cite{berezanskiy,kato,schechter} one can  find a technique
of proving selfadjointness based on computing the relative bound. The method presented
above is an alternative approach, based rather on the notions of commutativity and domination.
An example of a first order symmetric differential operator from \cite{Wojtylak08} shows the
difference between those two approaches.  

In the symmetric case (a0) the techinque presented in the theorem above was already used in the literature in the context of differential operators on manifolds  \cite{BMS,Gaffney} and graphs \cite{Golenia}. In \cite{dominate3} one can find examples of applications of the domination techniques to symmetric operators.  Therefore, in the present paper we do not focus our
attention on the (a0) class, but show possible applications of the main result in the classes (a1)--(a4).

 The content of the present paper is the following. Section \ref{s:P} has   a  preliminary character, but already in the consecutive section
  we prove the main result of the paper. In Section \ref{normalops} we consider the class (a1) of formally
normal operators, extending the results from \cite{Wojtylak08}. In Section \ref{matrixops} we
will consider $H$--symmetric operators (class (a4)) given by infinite matrices.

In Sections \ref{s:blizej} we  make a link with a theory of commutative domination in the
sense of \cite{poulsen,dominate1,Wojtylak11}.  Namely, the sequence $(T_n)_{n=0}^\infty$ in
Theorem \ref{intro} above may be in many cases chosen as
$$
T_n = n^m(S-{\mathrm i} \cdot n)^{-m},\quad n\in\mathbb{N},
$$
where $S$ is a selfadjoint operator and $m\geq 1$, examples can be found in the already
mentioned work \cite{dominate3}. However, this approach requires computing the commutator
\eqref{komintro}. In Theorem \ref{blizej} we replace \eqref{komintro} by a condition involving
the commutator $SA-AS$, the new assumption being stronger then \eqref{komintro}  but
nevertheless easier to calculate. Again, we formulate the result in the general setting of the
pair $(A_0,A)$, the case $SA-AS=0$ is the announced link with commutative domination. In
Section \ref{DO} we apply Theorem \ref{blizej} to a first order differential operator $A$ with
nonconstant coefficients. A necessary conditions, expressed in terms of coefficients, for $A$
being of class (a3)  and for $\mathcal{D}(\bar A)=\mathcal{D}(A^*)$ are provided. As the
reader noticed there are so far  no applications of the main result to the class (a2).

\section{Preliminaries}\label{s:P}

Through the whole paper $(\mathcal{K},\seq{\cdot,-})$ stands for a Hilbert space. The sum and
product of unbounded operators is understood in a standard way, see e.g. \cite{DS}. We put
    \[
    {\mathrm{ad}}(S,T):=ST-TS.
    \]
 We say that an operator
$S$ in $\mathcal{K}$ is {\it bounded} if $\norm{Sf}\leq c\norm f$ for  all $f\in\mathcal{D}(
S)$ and some $c\geq 0$.  We write $\mathbf{B}(\mathcal{K})$ for the space of all bounded
operators with domain equal $\mathcal{K}$,  stressing  the fact that \textit{not} every
bounded operator is in $\mathbf{B}(\mathcal{K})$.

Let $A$ be a closable, densely defined operator. We say that $A$ is \textit{symmetric} if
 $A\subseteq A^*$, \textit{selfadjoint} if $A=A^*$.

Let $q\in(0,+\infty)$, we say that $A$ is \textit{$q$--formally normal}  if
$\mathcal{D}(A)\subseteq\mathcal{D}(A^*)$ and  $\norm{A^*f}=\sqrt q\norm{Af}$ for
$f\in\mathcal{D}(A)$. We say that $A$ is \textit{$q$--normal} if $A$ is $q$--formally normal
and $\mathcal{D}(A)=\mathcal{D}(A^*)$. We refer the reader to \cite{Ota02,OtaSzafraniec07} for
a treatment on $q$--normals and related classes of operators. Note that (a2) together with
$\bar A_0=A^*$ gives  $q$--normality of $\bar A$. Indeed, since the graph norms of
$A_0=A^*\!\!\mid_{\mathcal{D}(A)}$ and $A$ are equivalent on $\mathcal{D}(A)$, we get
$\mathcal{D}({\bar A})=\mathcal{D}(A^*)$, i.e. $\bar A$ is $q$--normal. We call $A$
\textit{formally normal} (\textit{normal} ) if it is 1-formally normal (1-normal,
respectively).

Let $H\in\mathbf{B}(\mathcal{K})$ be selfadjoint and boundedly invertible. We say that $A$ is
\textit{$H$--symmetric}  if $A\subseteq H^{-1}A^*H$, \textit{$H$-selfadjoint} if  $A=
H^{-1}A^*H$. If we introduce an indefinite inner product on $\mathcal{K}$ by
$[f,g]=\seq{Hf,g}$, $f,g\in\mathcal{K}$, then  $(\mathcal{K},[\cdot,-])$  is a Krein space,
see \cite{AJ,bognar}. Defining $A^+$ as the adjoint of $A$ with respect to $[\cdot,-]$ we
easily see that $A^+=H^{-1}A^*H$. Hence, Theorem \ref{intro} can suite as a criterion for
selfadjointness of a closed symmetric operator in a Krein space, cf. \cite{Wojtylak08}.
Nevertheless, we will not use the indefinite inner product and the operator $A^+$ in the
present paper.

We also say that $A$ is \textit{essentially selfadjoint} (respectively,  essentially
$q$--normal, essentially $H$--selfadjoint) if $\bar A$ is selfadjoint (respectively,
 $q$--normal,  $H$--selfadjoint).

The following facts will be frequently used later on. If $S$ and $T$ are densely defined
operators in $\mathcal{K}$ and $ST$ is densely defined then $(ST)^*\supseteq T^*S^*$. If
additionally $S\in\mathbf{B}(\mathcal{K})$ then
\begin{equation}\label{ref1}
(ST)^*=T^*S^*.
\end{equation}

We say that an operator {\it $A$  dominates an operator $B$ on a linear space $\mathcal{E}$}
if $\mathcal{E}\subseteq\mathcal{D}(A)\cap\mathcal{D}( B)$ and there exists $c>0$ such that
$$
\norm{ Bf}\leq c(\norm{ Af} + \norm f),\quad f\in\mathcal{E}.
$$
If $\mathcal{E}=\mathcal{D}( B)$ then we  say that $A$ {\it dominates} $B$. Note that if both
operators are closed, then $A$ dominates $B$ if and only if $\mathcal{D}(
B)\subseteq\mathcal{D}(A)$, by the closed graph theorem.

\section{Approximate units for an unbounded operator. Main result. }\label{secd}
Let  $A$ be closable and densely defined, let  $A_0\subseteq A^*$ and let
$T\in\mathbf{B}(\mathcal{K})$. Consider the following conditions.
\begin{enumerate}
\item[(f1)]{\rm the  commutator ${{\mathrm{ad}}(T,\bar A)}$ is densely defined and bounded in $\mathcal{K}$;}
 \item[(f2)]{\rm     $T^*\mathcal{D}(A^*)\subseteq\mathcal{D}({\bar {A_0}})$ .}
 \end{enumerate}

If a sequence $(T_n)_{n=0}^\infty\subseteq\mathbf{B}(\mathcal{K})$ tends in the weak operator
topology to $I_\mathcal{K}$ and is such that each of the operators $T_n$ ($n\in\mathbb{N}$)
satisfies (f1,f2)  we will call it an  (f){\it--approximate unit for the pair $(A,A_0)$}. This
notion has some connections with quasicentral approximate units and the unbounded derivation,
see \cite{voiculescu} and the papers quoted therein.

\begin{prop}\label{komcond}
Let  $A$ be closable and densely defined, let  $A_0\subseteq A^*$ and let
$(T_n)_{n=0}^\infty\subseteq\mathbf{B}(\mathcal{K})$ be  an {\rm(f)}--approximate unit for
$(A,A_0)$. Then the following conditions are equivalent:
\begin{equation}\label{combdd}
\sup_{n\in\mathbb{N}}\norm{{{\mathrm{ad}}(T_n,\bar A)}}<+\infty;
\end{equation}
\begin{equation}\label{combddstar}
\sup_{n\in\mathbb{N}}\norm{\overline{{\mathrm{ad}}(T_n^*,A^*)}}<+\infty;
\end{equation}
    \begin{equation}\label{wotlim0}
    {\mathop{\rm WOT\ lim}_{n\to\infty}}\overline{{\mathrm{ad}}(T_n,\bar A)}=0;
    \end{equation}
\begin{equation}\label{wotlimstar}
    {\mathop{\rm WOT\ lim}_{n\to\infty}}\overline{{\mathrm{ad}}(T_n^*,A^*)}=0.
    \end{equation}
\end{prop}

\begin{proof}
 Fix $n\in\mathbb{N}$.  The operator ${\mathrm{ad}}(A^*,T_n^*)$ is
densely defined by (f2) and  is  contained in ${\mathrm{ad}}(T_n,\bar A)^*$. By (f1) the
operator ${\mathrm{ad}}(T_n,\bar A)^*$ belongs to $\mathbf{B}(\mathcal{K})$. Hence,
$$
\overline{{\mathrm{ad}}(A^*,T_n^*)}={\mathrm{ad}}(T_n,\bar A)^*.
$$
This shows the  equivalences $\eqref{combdd}\Leftrightarrow\eqref{combddstar}$ and $\eqref{wotlim0}\Leftrightarrow\eqref{wotlimstar}$.

Suppose now that  $\eqref{combdd}$ is satisfied. The weak convergence of $(T_n)_{n=0}^\infty$
to identity implies that for $f\in\mathcal{D}({{\mathrm{ad}}(T_n,\bar A)})$,
$g\in\mathcal{D}(A^*)$ one has
  $$
    \seq{{\mathrm{ad}}(T_n,\bar A) f,g}=\seq{\bar Af,T_n^*g}-\seq{T_nf,A^*g}\stackrel{n\to\infty}\longrightarrow \seq{\bar Af,g}-\seq{f,A^*g}=0.
 $$
Since $\mathcal{D}({{\mathrm{ad}}(T_n,\bar A)})$ and $\mathcal{D}(A^*)$ are dense in
$\mathcal{K}$  we have \eqref{wotlim0} by a standard triangle inequality argument.

The implication  $\eqref{wotlim0}\Rightarrow\eqref{combdd}$ holds, since every sequence
convergent in the weak operator topology is bounded in the norm, by the uniform boundedness
principle.

\end{proof}
After these preparations we can easily derive the main result of the paper.

\begin{thm}\label{main4}
Let  $A$ be closable and densely defined, let  $A_0\subseteq A^*$ and let
$(T_n)_{n=0}^\infty\subseteq\mathbf{B}(\mathcal{K})$ be  an {\rm(f)}--approximate unit for
$(A,A_0)$. If
 $$
\sup_{n\in\mathbb{N}}\norm{{{\mathrm{ad}}(T_n,\bar A)}}<+\infty,
$$
 then $\bar A_0=A^*$.
\end{thm}

\begin{proof}

Fix an arbitrary $f\in\mathcal{D}(A^*)$ and consider   the sequence $f_n:=T_n^*f$
($n\in\mathbb{N}$), which is contained in $\mathcal{D}({\bar A_0})$ by (f2). Observe that
\begin{equation}\label{weak1}
\seq{f_n,g}\to \seq{f,g}  \quad (n\to\infty),\qquad g\in\mathcal{K},
\end{equation}
since $T^*_n$ tends to $I_\mathcal{K}$ in the weak operator topology. Furthermore, note that
\begin{equation}\label{weak2}
\seq{A^*f_n,g}\to \seq{A^*f,g} \quad (n\to\infty),\qquad g\in\mathcal{K}.
\end{equation}
Indeed, since $f$ belongs to $\mathcal{D}(A^*)$, which is contained in
$\mathcal{D}({{\mathrm{ad}}(A^*,T_n^*)})$ by (f2), we have
    $$
    \seq{A^*T_n^*f,g}=\seq{T_n^*A^*f,g}+\seq{{{\mathrm{ad}}(A^*,T_n^*)}f,g},\quad g\in\mathcal{K}.
    $$
The first summand tends with $n\to \infty$ to $\seq{A^*f,g}$  by the convergence of
$T_n^*$, the second summand goes  to zero by Proposition \ref{komcond}.
Hence, \eqref{weak2} is shown.

Consider now the graph norm $\norm{\cdot}_{A^*}$ on $\mathcal{D}(A^*)$, which makes
$\mathcal{D}(A^*)$ a Banach space. Convergences \eqref{weak1} and \eqref{weak2} and the fact
that $f\in\mathcal{D}(A^*)$ was taken arbitrary imply that  $\mathcal{D}({\bar A_0})$ is
weakly dense in $(\mathcal{D}(A^*), \norm{\cdot}_{A^*})$. Since $\mathcal{D}({\bar A_0})$ is a
linear space, it is  dense in $\mathcal{D}(A^*)$ in the $\norm{\cdot}_{A^*}$--topology as
well. But $\mathcal{D}({\bar A_0})$ is closed in $\norm{\cdot}_{A^*}$--topology as
$A_0\subseteq A^*$. Hence, $\bar A_0=A^*$.

\end{proof}

\begin{remark}\label{r1} Observe that the following condition
\begin{enumerate}
\item[(e)] $\mathcal{R}( T)\subseteq\mathcal{D}(\bar A)$, $ \mathcal{R}(T^*)\subseteq\mathcal{D}({\bar{ A_0}})$
\end{enumerate}
implies (f1,f2). Indeed, if (e) holds then $\mathcal{D}(\bar
A)\subseteq\mathcal{D}({{\mathrm{ad}}(T,\bar A)})$. Furthermore, $\bar
AT\in\mathbf{B}(\mathcal{K})$, by the closed graph theorem.
 Since  $\mathcal{R}(T^*)\subseteq\mathcal{D}({\bar
A_0})\subseteq\mathcal{D}(A^*)$, we have $A^*T^*\in\mathbf{B}(\mathcal{K})$, again by the
closeness of the graph. By $(T\bar A)^*=A^*T^*$, the operator $T\bar A$ is bounded. Hence (f1)
is showed, (f2) is obvious. Therefore Theorem \ref{intro} is proved as well.
\end{remark}

\begin{remark}
 It was shown in \cite{Wojtylak08}  that in the
(a4) case conditions
\begin{enumerate}
\item[(d1)] the operators $T A$ and  $AT$  are bounded and the domain of the commutator
$\mathcal{D}({{\mathrm{ad}}(T,A)})$ is dense in $\mathcal{K}$;
\item[(d2)] the operator    $A_0T^*$  is densely defined.
\end{enumerate}
  (presented here in an equivalent form) imply (e), see Proposition 2 and the consecutive remarks. Hence, (d1,d2) together with (a4) imply  (f1,f2). Therefore,
 Theorem 3 of \cite{Wojtylak08} can be seen as a special case of Theorem \ref{main4} above.
\end{remark}

\section{Some normal operators}\label{normalops}

In this section we will concentrate on the (a1) class.  We begin with a proposition that
unifies Theorem 6 of \cite{Wojtylak08} and  Proposition 1 of \cite{nussbaum69}. If $E$ is the
spectral measure of a normal operator $N$ and $\mathbb{D}$ is the closed unit disc then we set
$$
\mathcal{B}(N):=\bigcup_{n\in\mathbb{N}}\mathcal{R}(E(n\mathbb{D})).
$$

\begin{prop}
 Let $\mathcal{K}$ be a Hilbert space, and let $A$ be a formally normal operator in $\mathcal{K}$. If there
exists a normal operator $N$ in $\mathcal{K}$ such that
$\mathcal{B}(N)\subseteq\mathcal{D}(A)$ and the spectral measure $E$ of $N$ satisfies the
condition
                            $$
                             \sup_{n\in\mathbb{N}} \norm{{\mathrm{ad}}(A, E(n\mathbb{D}))} <+\infty
                                   $$
then $A$ is essentially normal.
\end{prop}

\begin{proof}
We set $T_n:=E(n\mathbb{D})$ ($n\in\mathbb{N}$) and apply Theorem \ref{intro}.
\end{proof}
Next let us provide an  analogue of Theorem 7 of \cite{Wojtylak08}, see also there for
references to works on selfadjoint Dirac operators. Take the Hilbert space
$\mathcal{K}:=(L^2({\mathbb{R}}^m))^k$, where $k,m\in\mathbb{N}$ and let ${\mathcal
C^\infty_0({\mathbb{R}}^m)}$ denote the complex space of infinite differentiable functions on
${\mathbb{R}}^m$ with compact supports. Consider  the differential operator $A$ in
$\mathcal{H}$ given by
    \[
     Au:= {\mathrm i}^{-1}\sum_{l=1}^m \alpha_l\frac{\partial
    u}{\partial x_l}+Qu,\quad u\in\mathcal{D}({A})={(\mathcal C^\infty_0({\mathbb{R}}^m))^k},
    \]
where $\alpha_1,\dots, \alpha_m$ are complex $k\times k$ matrices and
$Q:{\mathbb{R}}^m\to\mathbb{C}^{k\times k}$ is a locally integrable matrix--valued function.
Note that ${(\mathcal C^\infty_0({\mathbb{R}}^m))^k}\subseteq\mathcal{D}(A^*)$ and
 $$
A^*u= {\mathrm i}^{-1}\sum_{l=1}^m \alpha^*_l\frac{\partial
    u}{\partial x_l}+Q^*u,\quad u\in{(\mathcal C^\infty_0({\mathbb{R}}^m))^k}.
 $$
A direct calculation shows that the following conditions
\begin{equation}\label{Afnorm}
\begin{array}{rcll}
 \alpha_l^*\alpha_r&=&\alpha_l\alpha_r^*, \qquad&\text{for }  r,l=1,\dots, m;\\
Q(x)Q^*(x)&=&Q^*(x)Q(x)\qquad &\text{for a.e. } x\in{\mathbb{R}}^m;\\
\alpha_l^*Q(x)&=&\alpha_lQ^*(x),\qquad &\text{for a.e. } x\in{\mathbb{R}}^m,\ l=1,\dots, m.
\end{array}
\end{equation}
imply $\norm{Au}=\norm{A^*u}$ for $u\in{(\mathcal C^\infty_0({\mathbb{R}}^m))^k}$, i.e. formal
normality of $A$. We say that $Q$ \textit{satisfies the local H\"older condition}
 if for every $n\in\mathbb{N}$ there exists a $b_n\in(0, 1]$ such that
$$
                                  \sup_{|x|,|y|\leq n, \ x\neq y}  \frac{     |Q(x) - Q(y)|   }{|x-y|^{b_n}}           < \infty,
$$
where $|\cdot|$ denotes the Euclidean norm on ${\mathbb{R}}^m$ and ${\mathbb{R}}^k$.

\begin{prop}\label{normaldirac}
Assume that conditions \eqref{Afnorm} hold  and that the function $Q$ satisfies the local
H\"older condition. Then  $A$ is essentially normal in $\mathcal{K}$.
\end{prop}

\begin{proof}[Sketch of the proof]
We apply Theorem \ref{intro} to the (a1) instance.  The construction
of the sequence $T_n$ follows exactly the same lines as in the proof of Theorem 7 of
\cite{Wojtylak08}.
\end{proof}

\section{Infinite $H$--selfadjoint matrices}\label{matrixops}

In  \cite{dominate3}  Cicho\'n, Stochel and Szafraniec investigated symmetric integral and
matrix operators. The main tools were the domination techniques from their previous paper
\cite{dominate2} based on the computation of the first and second commutator. 
The discussion on applicability of these criteria in the Jacobi matrix case can be found in \cite{dominate2}, in the present work we will show how
the first commutator reasonings  can be applied to $H$--symmetric operators, restricting  to
the  matrix operators on  $\ell^2=\ell^2(\mathbb{N})$ ($\mathbb{N}=\left\{1,2,\dots\right\})$. 
By $\ell_0^2$ we denote the space of all complex sequences with finite number of nonzero
entries.
 Given a matrix $[a_{k,l} ]_{k,l\in\mathbb{N}}$ , we define the matrix operator $\tilde A$ by
 $$
\mathcal{D}({\tilde A})=\left\{ \{\xi_k \}_{k\in\mathbb{N}} \in \ell^2: \sum_{k\in
\mathbb{N}}\left| \sum_{l\in\mathbb{N}} |a_{k,l} \xi_l |\right|^2 < +\infty  \right\},
$$
$$
\tilde A\{\xi_k\}_{k\in\mathbb{N}}= \left\{ \sum_{l\in\mathbb{N}} a_{k,l}
\xi_l\right\}_{k\in\mathbb{N}}.
$$
Let us suppose that the matrices $[h_{k,l}]_{k,l\in \mathbb{N}}$ and
$[g_{k,l}]_{k,l\in\mathbb{N}}$ have the following properties:
\begin{itemize}
\item[(h1)] $[h_{k,l}]_{k,l\in\mathbb{N}}$ and $[g_{k,l}]_{k,l\in\mathbb{N}}$ are hermitian--symmetric matrices;
\item[(h2)] $[g_{k,l}]_{k,l\in \mathbb{N}}$ is a band matrix, i.e.
there exists a $p\in\mathbb{N}$ such that $g_{k,l}=0$ for $|k-l|>p$ ;
\item[(h3)]  $s_g:=\sup_{k,l\in \mathbb{N}}
|g_{k,l}|<+\infty$ and $[h_{k,l}]_{k,l\in\mathbb{N}}$ defines a bounded operator;
\item[(h4)] $\sum_{j\in \mathbb{N}}h_{k,j}g_{j,l}=\sum_{j\in \mathbb{N}} g_{k,j}h_{j,l}=\delta_{k,l}$ for $k,l\in \mathbb{N}$.
\end{itemize}
Then $[h_{k,l}]_{k,l\in \mathbb{N}}$ and $[g_{k,l}]_{k,l\in \mathbb{N}}$  define bonded,
selfadjoint operators on $\ell^2$, which will be called $G$ and $H$, respectively; obviously
$G=H^{-1}$. An example of such a matrix $[h_{k,l}]_{k,l\in \mathbb{N}}$, additionally equal to
$[g_{k,l}]_{k,l\in\mathbb{N}}$, is a block--diagonal matrix with each block on the diagonal
being of the anti--diagonal form
$$
\pm\begin{pmatrix}0&\dots& 1\\ \vdots& & \vdots \\ 1 &\dots& 0 \end{pmatrix}
$$
not necessarily of the same size. The proposition below is an $H$--symmetric version of
Theorem 13 of \cite{dominate3}.

\begin{prop}\label{matrix}  Let $\{c_n \}_{n\in \mathbb{N}}$ be a sequence
of real numbers and $m\geq 0$ be an integer such that
       the matrices
\begin{equation}\label{M1}
\left[ \frac{|a_{k,l+q}|}{1+|c_l|^m} \right]_{k,l\in \mathbb{N}},\qquad q\in\left\{-p,\dots,
p\right\}
\end{equation}
$($with $a_{k,r}:=0$ for $r\leq0)$ and
\begin{equation}\label{M2}
\left[ |a_{k,l}|\frac{|c_k-c_l|}{1+|c_k|+|c_l|} \right]_{k,l\in \mathbb{N}}
\end{equation}
define  bounded operators on   $\mathcal{K}$. If
\begin{equation}\label{AG}
\sum_{q=-p}^p a_{k,l+q}g_{l+q,l}=\sum_{q=-p}^p  g_{k,k+q}\bar a_{l,k+q}, \quad
k,l\in\mathbb{N},
\end{equation}
then the operator  $A=\tilde A\!\!\mid_{\ell_0^2(\mathbb{N})}$ is essentially $H$--selfadjoint
and $\tilde A=\bar A$.
\end{prop}
Note that \eqref{M1} implies that $\ell_0^2\subseteq\mathcal{D}(A)$ and condition \eqref{AG}
obviously means that $A$ is $H$--symmetric.   For the proof  of the proposition we will need
the following lemma, also to be used in the next section. As usually,  $\rho(S)$ stands for the
resolvent set of $S$.

\begin{lem}\label{komm1}
Let $A$ be a closable, densely defined operator,  let $S$ be a closed densely defined operator
and let $z\in\rho(S)$. If
 $$
 \mathcal{D}({S^m})\subseteq\mathcal{D}(\bar A),\quad \mathcal{D}({S^{*m}})\subseteq\mathcal{D}(A^*)\ \ \textrm{ for some }m\in\mathbb{N},
$$
then ${\mathrm{ad}}((S-z)^{-1},\bar{ A})$ is a closable densely defined operator. If it is
additionally
 bounded then
    \begin{equation}\label{orazoraz}
\norm{ {\mathrm{ad}}((S-z)^{-m},\bar A)}\leq  m\norm{(S-z)^{-1}}^{m-1}\norm{{\mathrm
ad}((S-z)^{-1},\bar A)}.
    \end{equation}
\end{lem}

\begin{proof}
First note that since $z\in\rho(S)$, one has $\bar z\in\rho(S^*)$. Hence, the operators $S^m$
and $S^{*m}$ are closed and densely defined with nonempty resolvent sets
\cite[Thm.VII.9.7]{DS}. Since
$(S-z)^{-1}\mathcal{D}({S^m})=\mathcal{D}({S^{m+1}})\subseteq\mathcal{D}(\bar A)$, the
commutator ${\mathrm{ad}}((S-z)^{-1},\bar A)$ is densely defined. Furthermore, note that
    $$
    {\mathrm{ad}}((S-z)^{-1},\bar{ A})^*\supseteq{\mathrm{ad}}( A^*,(S^*-\bar z)^{-1}).
    $$
The domain of the operator on the right hand side contains $\mathcal{D}({S^{*m}})$, which is
dense in $\mathcal{K}$. By von Neumann's theorem ${\mathrm{ad}}((S-z)^{-1},\bar{ A})$ is
closable. Suppose now that it is also bounded. Since  $\mathcal{D}({S^m})$ is dense in
$\mathcal{K}$, the formula (\cite[Prop.2(i)]{dominate2})
    $$
    {\mathrm{ad}}((S-z)^{-m},\bar A)f=\sum_{j=0}^{m-1} (S-z)^{-j}{\mathrm{ad}}((S-z)^{-1},\bar A)(S-z)^{-m+1+j}f,
    \quad f\in\mathcal{D}({S^m})
    $$
 gives the desired estimate.
\end{proof}

\begin{proof}[Proof of Proposition \ref{matrix}]
First note that by \eqref{M1} we have $\ell_0^2\subseteq \mathcal{D}(A)$. Now define the
selfadjoint operator $S$ by the diagonal matrix $[\delta_{k,l}c_l]_{k,l\in \mathbb{N}}$. By
\eqref{M1} with $q=1$ we obtain that $A(S^m-z)^{-1}$ is a bounded operator, hence $\bar
A(S^m-z)^{-1} \in\mathbf{B}(\ell^2)$ and consequently
\begin{equation}\label{RA}
\mathcal{D}({S^m})\subseteq\mathcal{D}(\bar A).
\end{equation}
Let now $\left\{\xi_l\right\}_{l\in \mathbb{N}}\in\ell_0^2$, then $G(S^{m}-\bar
z)^{-1}\left\{\xi_l\right\}_{l\in \mathbb{N}}\in\ell_0^2=\mathcal{D}(A)$ and
$$
AG(S^{m}-\bar z)^{-1}\left\{\xi_l\right\}_{l\in \mathbb{N}}= AG\left\{\frac{ \xi_l}{c_l^m-\bar
z}\right\}_{l\in \mathbb{N}}=
$$ $$
\left\{\sum_{k\in\mathbb{N}} a_{r,k}\sum_{|k-l|\leq p}g_{k,l}\frac{\xi_l}{c_l^m-\bar
z}\right\}_{r\in \mathbb{N}}= \sum_{q=-p}^p \left\{\sum_{k\in\mathbb{N}} a_{r,l+q}\
g_{l+q,l}\frac{\xi_{l}}{c_{l}^m-\bar z}\right\}_{r\in \mathbb{N}}.
$$
It follows now easily from \eqref{M1} and the assumption (h3) that the operator
$C:=AG(S^{m}-\bar z)^{-1}$ is bounded on $\ell_0^2$.  Since $C^*\supseteq (S^{*m}-\bar
z)^{-1}GA^*$, $C$ is closable. Hence, $\bar AG(S^{m}-\bar z)^{-1}=\bar C\in\mathbf{B}(\ell^2)$
and consequently
\begin{equation}\label{RA*}
\mathcal{D}({S^m})\subseteq\mathcal{D}({\bar A_0}),
 \end{equation}
where $A_0= HAG$, according to (a4).
This together with \eqref{RA} implies that assumption \eqref{inclran} is satisfied with
 $$
 T_n=n^m(S-n{\mathrm i})^{-m},\qquad n\in\mathbb{N}.
 $$
Obviously, $T_n$ tends with $n$ to $I_{\ell^2}$ in the strong operator topology.
 To apply Theorem \ref{intro}  one needs to show that $(T_n)_{n\in\mathbb{N}}$ and $A$ satisfy
\eqref{komintro}. Observe that for $\xi=\{\xi_k\}_{k\in\mathbb{N}} \in\ell_0^2$ one has
$$
{\mathrm{ad}}((S-n {\mathrm i})^{-1},\bar A)\xi=\left\{ \sum_{l\in\mathbb{N}}\frac{  a_{kl}
\xi_l(c_k-c_l)}{(n{\mathrm i}-c_k)(n{\mathrm i}-c_l)}  \right\}_{k\in\mathbb{N}}.
$$
Since (cf. \cite{dominate3} p.769)
$$
\frac{ n}{|n{\mathrm i}-c_k||n{\mathrm i}-c_l|}\leq \frac{\sqrt 3}{1+|c_k|+|c_l|},\quad
n,k,l\in\mathbb{N}
$$
we conclude that
$$
n\norm {{\mathrm{ad}}(S-n{\mathrm i})^{-1}, A)\xi }^2\leq 3 \norm{K}^2 \norm{\xi} ^2, \quad
\xi\in\ell_0^2,\ n\in\mathbb{N},
$$
where $K$ is the bounded operator given by \eqref{M2}. Thanks to \eqref{RA} and \eqref{RA*} we
can apply Lemma \ref{komm1} and obtain that  the commutator ${\mathrm{ad}}((S-n{\mathrm
i})^{-1},\bar A)$ is closable. Hence, it is bounded and
$$
 n \norm{{\mathrm{ad}}((S-n\ {\mathrm i})^{-1},\bar A)}\leq \sqrt 3\norm K ,\quad n\in\mathbb{N}.
$$
By the second part of  Lemma \ref{komm1}
$$
\sup_{n\in\mathbb{N}} n^m \norm{{\mathrm{ad}}((S-n{\mathrm i})^{-m},\bar A)}<+\infty,
$$
which is the desired inequality \eqref{komintro}. Applying Theorem \ref{intro} we get $\bar A_0=A^*$, i.e. the operator $A$ is essentially $H$--selfadjoint. Since $\tilde A$ is
$H$--symmetric and contains $\bar A$, one has $\tilde A=\bar A$.
\end{proof}

\begin{prop}
 Suppose that we are
given real numbers $d \geq 0$, $s\geq 0$ , $\alpha > 2$.  If \eqref{AG} is satisfied and
 \begin{equation}\label{modakl}
|a_{k,l}|\leq \left\{ \begin{array}{ll} d(1 + k + l)/(|k - l|^\alpha), & k\neq l\\
                                       d(k + 1)^s, k=l,\end{array} \right. ,\quad k,l\geq 0
\end{equation}
the operator  $A=\tilde A\!\!\mid_{\ell_0^2(\mathbb{N})}$ is essentially $H$--selfadjoint and
$\tilde A=\bar A$.
\end{prop}

This proposition has  again its symmetric origin in \cite{dominate3}, namely  of Proposition
14.  Note that besides the assumption of $H$--symmetry in \eqref{AG} the matrices
$[g_{kl}]_{kl\in\mathbb{N}}$ and $[g_{kl}]_{kl\in\mathbb{N}}$ are not involved in the
assumptions.

\begin{proof}
We need to show that (cf. \cite{dominate3})
\begin{equation}\label{sum}
\sum_{k,l\in\mathbb{N}}\frac{|a_{k,l+q}|^2}{(1+|c_l|^m)^2}<+\infty,\quad q=-p,\dots, p,
\end{equation}
with $c_l=l$, which will guarantee boundedness of all operators in \eqref{M1}. It was shown in
\cite{dominate3} that
$$
\sum_{k\in\mathbb{N}} |a_{k,l}|^2\leq \mathcal{O}(l^2+l^{2s}).
$$
Hence,
$$
\sum_{k\in\mathbb{N}} |a_{k,l+q}|^2\leq
\mathcal{O}((l+q)^2+(l+q)^{2s})=\mathcal{O}(l^2+l^{2s}),\quad q=-p,\dots, p
$$
and \eqref{sum} holds with $m>s+3/2$.
  Boundedness of the
operator in \eqref{M2} follows the same lines as in the proof of Proposition 14 of
\cite{dominate3}.
\end{proof}

\section{Towards commutative domination}\label{s:blizej}

In this section we will show a relation between the results on commutative
(\cite{dominate1,Wojtylak05,Wojtylak11}) and noncommutative domination
(\cite{dominate2,dominate3,Wojtylak07,Wojtylak08}). One should mention here the work by Nelson
\cite{nelson}, which deals with the symmetric case and analytic vectors. Nevertheless, the aim
of the present paper is to consider classes different then symmetric operators using  simple
graph arguments only. We say that $\mathcal{E}\subseteq\mathcal{D}( S)$ is a {\it core for
$S$} if the graph of $S$ is contained in the closure of the graph $S\!\!\mid_{\mathcal{E}}$. The sybol $\mathcal{D}^{\infty}(A)$ stands for 
$\bigcap_{n=0}^\infty \mathcal{D}(A^n)$.

\begin{thm}\label{blizej}
Let $A$ be a closable, densely defined operator, let $A_0\subseteq A^*$ and let $S$ be a
closed densely defined operator  such that there exists a sequence
$(z_n)_{n=0}^\infty\subseteq\rho(S)$ satisfying
\begin{equation}\label{Omega}
{\mathop{\rm WOT\ lim}_{n\to\infty}}z_n(S-z_n)^{-1}=I_\mathcal{K}.
\end{equation}
Assume that
\begin{itemize}
\item[(i)] $\mathcal{D}^\infty( S)\subseteq\mathcal{D}({ \bar A})$, $\mathcal{D}^\infty({S^*})\subseteq\mathcal{D}({\bar A_0})$,
\item[(ii)] ${\mathrm{ad}}(A_0,S^*)$ is densely defined,
\item[(iii)] there exists a linear subspace $\mathcal{D}\subseteq\mathcal{D}({{\mathrm{ad}}(S,\bar A)})$, which is
a core for $S$ and  $S$ dominates ${\mathrm{ad}}(S,\bar A)$ on
 $\mathcal{D}$,
\end{itemize}
 then $\bar A_0=A^*$. If, additionally, the resolvent set of $A$ is nonempty and
\begin{itemize}
\item[(i')] $\mathcal{D}({S})\subseteq\mathcal{D}(\bar A)$
\item[(iii')] ${\mathrm{ad}}(S,\bar A)f=0$ for $f\in\mathcal{D}$,
 \end{itemize}
 then the resolvents of $A$ and $S$ commute.
 \end{thm}

The problem of existence of  a sequence $(z_n)_{n=0}^\infty$ satisfying \eqref{Omega} was
discussed in \cite{Wojtylak08} in detail. In case $S$ is  (similar to) a selfadjoint operators
in Hilbert spaces such a sequence  exists. Note that the precise knowledge of the sequence is
not necessary to apply the theorem.

\begin{proof}
By assumption (ii) and von Neumann's theorem we get ${\mathrm{ad}}(S,A)$ closable. Standard
domination technique (see e.g. Lemma 1 of \cite{Wojtylak08}) gives
 $$
 \mathcal{D}( S)=\mathcal{D}\left({\overline{S\!\!\mid_{\mathcal{D}}}}\right) \subseteq\mathcal{D}\left({\overline{{\mathrm{ad}}(S,\bar A)}}\right).
 $$
  Hence, $S$ dominates ${\mathrm{ad}}(S,\bar A)$,
i.e. for some  $c\geq 0$ we have
\begin{equation}\label{[S,A]f}
\norm{\overline{{\mathrm{ad}}(S,\bar A)}f}\leq c\big(\norm f+\norm{ Sf}\big), \quad
f\in\mathcal{D}({S}).
\end{equation}
We apply \eqref{[S,A]f} to $f:=(S-z_n)^{-1}g\in\mathcal{D}( S)$
 with arbitrary $n\in\mathbb{N}$ and $g\in\mathcal{K}$, getting
\begin{equation}\label{[SA]}
 \begin{array}{c}
\norm{\overline{{\mathrm{ad}}(S,\bar A)}(S-z_n)^{-1}g}\leq c\left(\norm{(S-z_n)^{-1}g}+\norm{S(S-z_n)^{-1}g}\right)\leq\\
\leq  c\left(\norm{(S-z_n)^{-1}}+\norm{z_n(S-z_n)^{-1}+I}\right)\norm g.
 \end{array}
 \end{equation}
 It is now apparent that  there exists a constant $d\geq 0$, such that
\begin{equation}\label{dszac}
\norm{\overline{{\mathrm{ad}}(S,\bar A)}(S-z_n)^{-1}}\leq d ,\quad n\in\mathbb{N}.
\end{equation}
Fix $z\in\rho(S)$, then
    $$
    {\mathrm{ad}}((S-z)^{-1} ,\bar A)
    \supseteq(S-z)^{-1}\bar A (S-z) (S-z)^{-1}     -   (S-z)^{-1}(S-z)  \bar A(S-z)^{-1}=
    $$ $$
                      =(S-z)^{-1}{\mathrm{ad}}(\bar A,(S-z))(S-z)^{-1} =
                      (S-z)^{-1}{\mathrm{ad}}(\bar A,S)(S-z)^{-1}=:C.
                      $$
By \eqref{[SA]} the operator $C$ is bounded, furthermore, it  is also densely defined. Indeed,
the linear space $\mathcal{F}=(S-z)\mathcal{D}$ is contained in $\mathcal{D}( C)$ because
$(S-z)^{-1}\mathcal{F}=\mathcal{D}\subseteq\mathcal{D}({{\mathrm{ad}}(\bar A,S)})$ and
$\mathcal{F}$ is dense in $\mathcal{K}$ because $z\in\rho(S)$ and $\mathcal{D}$ is a core for
$S$.

 By Proposition 8.1 of
\cite{Wojtylak07} there exists $m\in\mathbb{N}$ such that
$\mathcal{D}({S^m})\subseteq\mathcal{D}({ \bar A})$ and
$\mathcal{D}({S^{*m}})\subseteq\mathcal{D}({\bar A_0})$. By Lemma \ref{komm1} the commutator
${\mathrm{ad}}((S-z)^{-1} ,\bar A)$ is closable. Since it contains the densely defined and
bounded operator $C$,  its closure belongs to $\mathbf{B}(\mathcal{K})$.  By \eqref{dszac} we
have
  \begin{equation}\label{c1}
    |z_n|\norm{{\mathrm{ad}}((S-z_n)^{-1} ,\bar A)}\leq|z_n|\norm{(S-z_n)^{-1}}\norm{{\mathrm{ad}}(\bar A,S)(S-z_n)^{-1}}\leq td,
    \end{equation}
with $t=\sup_{n\in\mathbb{N}}  \norm{z_n(S-z_n)^{-1} }$, which is finite because of
\eqref{Omega}. By the second part of Lemma \ref{komm1} we  have
\begin{equation}\label{supkolejne}
 \sup_{n\in\mathbb{N}}|z_n|^m\norm{{\mathrm{ad}}((S-z_n)^{-m},\bar A)}<\infty.
 \end{equation}
By  of Theorem \ref{intro} applied to  $T_n=z_n^m(S-z_n)^{-m}$
 we get $\bar A_0=A^*$.

To prove the second statement of the theorem fix $z\in \rho(S)$ and $w\in\rho(A)$. One can
easily check, that (iii') implies that $Cf=0$  for $f\in\mathcal{D}( C)$, consequently
$\overline{{\mathrm{ad}}((S-z)^{-1},A)}=0$.  Observe that
$$
(A-w)^{-1}{\mathrm{ad}}((S-z)^{-1},A)(A-w)^{-1}={\mathrm{ad}}((A-w)^{-1},(S-z)^{-1}),
$$
where both operators are in $\mathbf{B}(\mathcal{K})$ by (i'). In consequence both of them are
zero.
\end{proof}

\section{Differential operators}\label{DO}

As an application of Theorem \ref{blizej} consider the  differential operator
\[
     Au:= {\mathrm i}^{-1}\sum_{l=1}^m Q_l\frac{\partial
    u}{\partial x_l},\quad u\in\mathcal{D}({A})={(\mathcal C^\infty_0({\mathbb{R}}^m))^k},
    \]
in the Hilbert space $\mathcal{K}=(L^2({\mathbb{R}}^m))^k$ ($k,m\in\mathbb{N}$). We assume
that
 $Q_1,\dots, Q_m:{\mathbb{R}}^m\to\mathbb{C}^{k\times k}$ are $\mathcal{C}^2$--functions. First let us also note, that if $P_1,P_2$ are complex polynomials of $m$ variables then the operator
$P_1(\frac{\partial u}{\partial x_1},\dots, \frac{\partial u}{\partial x_m})$ dominates
$P_2(\frac{\partial u}{\partial x_1},\dots, \frac{\partial u}{\partial x_m})$ on ${(\mathcal
C^\infty_0({\mathbb{R}}^m))^k}$ if and only if for some $c>0$
\begin{equation}\label{diffdomin}
|P_2(\zeta)|\leq c( 1+ |P_1(\zeta)|),\quad \zeta\in{\mathbb{R}}^m.
\end{equation}
Indeed, the case $k=1$ is well known (see e.g. \cite{Limanskii04}) and the multidimensional
case is a simple consequence of the one--dimensional one.  For other types  of domination
inequalities for differential operators we refer the reader to
\cite{Hormander55,LimanskiiMalamud08} and the papers quoted therein.    Let us introduce the
following notation
$$
Q(x)=(Q_r^*(x)Q_l(x))_{r,l=1}^m\in\mathbb{C}^{mk\times mk},\qquad x\in{\mathbb{R}}^m,
$$
$$
Q^{(*)}(x)=(Q_r(x)Q_l^*(x))_{r,l=1}^m\in\mathbb{C}^{mk\times mk},\qquad x\in{\mathbb{R}}^m.
$$

\begin{prop}
 If
 \begin{equation}\label{QL}
 \frac{\partial Q_j }{\partial x_i},\frac{\partial^2 Q_j }{\partial x_i x_h},  \in
L^\infty({\mathbb{R}}^m),\quad h,i,j=1,\dots, m
\end{equation}
and for some $c_1>0$ one has
\begin{equation}\label{QQ}
c_1^{-1} Q(x)\leq  Q^{(*)}(x) \leq c_1 Q(x) 
,\quad x\in{\mathbb{R}}^m
\end{equation}
and for some $c_2>0$
\begin{equation}\label{QI}
 Q(x) \leq c_2 I_{\mathbb{C}^{mk\times mk}}
,\quad x\in{\mathbb{R}}^m,
\end{equation}
 then $\mathcal{D}(\bar A)=\mathcal{D}(A^*)$.
\end{prop}
\begin{proof}

 First we will show that the graph norms of $A$ and
$A^{*}$ are equivalent on ${(\mathcal C^\infty_0({\mathbb{R}}^m))^k}$. Denoting by
$\seq{\cdot,-}$ the standard inner product
 in $\mathbb{C}^k$ and $\mathbb{C}^{mk}$ and setting
 $$
\partial u:=\left(\frac{\partial u_1 }{\partial x_1},\dots,\frac{\partial u_k }{\partial x_1},\dots,\frac{\partial u_1 }{\partial x_m},\dots,\frac{\partial u_k }{\partial x_m}\right)
$$
one has
$$
\norm{Au}^2=\int_{{\mathbb{R}}^m}\sum_{l,r=1}^m\seq{Q_r^*(x)Q_l(x) \frac{\partial
    u}{\partial x_l}(x) ,\frac{\partial
    u}{\partial x_r}(x) }dx=
$$\begin{equation}\label{QA}
\int_{{\mathbb{R}}^m} \seq{ Q(x) \partial u(x),\partial u(x) }dx.
\end{equation}
Furthermore, note that
$$
A^*u=\mathrm{i}^{-1}\sum_{l=1}^m \frac{\partial}{\partial x_l} Q_l^*u=
\mathrm{i}^{-1}\sum_{l=1}^m \frac{\partial Q_l^*}{\partial x_l} u+\mathrm{i}^{-1}\sum_{l=1}^m Q_l^*\frac{\partial u}{\partial x_l} u.
$$
By \eqref{QL} the first summand on the right hand side is a bounded operator of $u$.  Thus the graph norms  of $A^*$ and $B=\sum_{l=1}^m Q_l^*\frac{\partial u}{\partial x_l}$ are equivalent on $(\mathcal{C}_0(\mathbb{R}^m))^k$. Furthermore, for $u\in (\mathcal{C}_0(\mathbb{R}^m))^k$ one has 
$$
\norm{Bu}^2=\int_{{\mathbb{R}}^m}\sum_{l,r=1}^m\seq{Q_r(x)Q^*_l(x) \frac{\partial
    u}{\partial x_l}(x) ,\frac{\partial
    u}{\partial x_r}(x) }dx=
$$ $$
\int_{{\mathbb{R}}^m} \seq{ Q^{(*)}(x) \partial u(x),\partial u(x) }dx,
$$
which, together with \eqref{QQ} and \eqref{QA} implies the equivalence of graph norms of  $B$ and $A$, and hence $A^*$ and $A$, on $(\mathcal{C}_0(\mathbb{R}^m))^k$.

Consider  the essentially selfadjoint, nonnegative operator
$$
Su=-\frac{\partial^2 u }{\partial x^2_1}-\cdots-\frac{\partial^2 u }{\partial x^2_m},\quad
u\in{(\mathcal C^\infty_0({\mathbb{R}}^m))^k}=\mathcal{D}( S).
$$
Note that by \eqref{QI} and \eqref{QA}  one has 
$$
\norm{Au}^2\leq c_2  \int_{{\mathbb{R}}^m} \seq{  \partial u(x), \partial u(x) }dx =  \seq{Su,u}= \norm{S^{1/2} u}^2.
$$
Therefore $S^{1/2}$, and in consequence $S$, dominates $A$ on ${(\mathcal C^\infty_0({\mathbb{R}}^m))^k}$. Furthermore,
$$
{\mathrm i}\cdot {\mathrm{ad}}(S,A)u=-\sum_{l,r=1}^m\frac{\partial^2}{\partial
x^2_l}\left(Q_r\frac{\partial
    u}{\partial x_r}\right)+ \sum_{l,r=1}^m Q_l\frac{\partial
    }{\partial x_l}\left(\frac{\partial^2u}{\partial x^2_r}\right)=
$$ $$
-\sum_{l,r=1}^m\frac{\partial Q_r}{\partial x_l}\frac{\partial^2    u}{\partial x_l
x_r}-\sum_{l,r=1}^m\frac{\partial^2Q_r}{\partial x^2_l}\frac{\partial
    u}{\partial x_r}.
$$
Application of \eqref{diffdomin} to the right hand side of the inequality
$$
\norm{{\mathrm{ad}}(S,A)u}\leq\sum_{l,r=1}^m\norm{\frac{\partial Q_r}{\partial
x_l}}_{L^\infty({\mathbb{R}}^m)}\norm{\frac{\partial^2
    u}{\partial x_l x_r}}+\sum_{l,r=1}^m\norm{\frac{\partial^2Q_r}{\partial x^2_l}}_{L^\infty({\mathbb{R}}^m)}\norm{\frac{\partial
    u}{\partial x_r}}.
$$
shows that $S$ dominates ${\mathrm{ad}}(S,A)$ on ${(\mathcal C^\infty_0({\mathbb{R}}^m))^k}$.
Hence, by Theorem \ref{blizej} we get $\mathcal{D}(\bar A)=\mathcal{D}(A^*)$.
\end{proof}

\section*{Acknowledgment}

The author is indebted to Doctor Dariusz Cicho\'n, Professor Jan Stochel and Professor
Franciszek Hugon Szafraniec for an inspiration and valuable discussions.

\end{document}